\documentclass[12pt,epsfig]{article}

\usepackage{amsmath, amssymb, amsthm}
\usepackage[margin=1.4 in]{geometry}

\usepackage[pdftex]{graphicx}

\numberwithin{equation}{section}
\newtheorem{thm}{Theorem}[section]
\newtheorem{lem}{Lemma}[section]
\newtheorem{rem}{Remark}[section]
\newtheorem*{rem*}{Remark}

\title 
{The operator-splitting method for Cahn-Hilliard is stable}
\date{}

\author
{Dong Li
\thanks
{Department of Mathematics, the Hong Kong University of Science
\& Technology, Clear Water Bay, Hong Kong.
 Email: {mpdongli@gmail.com}.
 }\qquad
{Chaoyu Quan}	
\thanks{SUSTech International Center for Mathematics, Southern University of Science and Technology,	Shenzhen, China.
Email: {quancy@sustech.edu.cn}.
}
}

\begin{document}
\maketitle
\begin{abstract}
We prove energy stability of a standard operator-splitting method for the Cahn-Hilliard equation. 
We establish uniform bound of Sobolev norms of the numerical solution and convergence
of the splitting approximation.  This is the first unconditional energy stability result for the operator-splitting method for the Cahn-Hilliard equation. Our analysis can be extended to many other models.
\end{abstract}
\section{Introduction}
We consider  numerical solutions of  the Cahn-Hilliard (\cite{CH58})
equation:
\begin{align} \label{1}
\begin{cases}
\partial_t u  = \Delta ( -\nu \Delta u +f (u) ), \quad (t,x) \in (0, \infty) \times \Omega, \\
u \Bigr|_{t=0} =u_0,
\end{cases}
\end{align}
where $u=u(t,x)$ is a real-valued function corresponding to the concentration difference
in a binary system. The parameter  $\nu>0$ is usually called the mobility coefficient which is taken
to a constant here for simplicity. The nonlinear term
$f(u)$ is derived from a standard double well potential, namely:
\begin{align*}
f(u)=u^3- u = F^{\prime}(u), \quad F(u) = \frac 14 (u^2-1)^2.
\end{align*}
Due to this specific choice of equal-well double potential, 
the minima of the potential are located at $u=\pm 1$ which corresponds
different phases or states. The length scale of the transitional region is usual proportional to $\sqrt{\nu}$. In this note we  take the spatial domain $\Omega$ in \eqref{1} as the  two-dimensional $2\pi$-periodic torus $\mathbb T^2=\mathbb R^2/ 2\pi \mathbb Z^2
=[-\pi,\pi]^2$.  Our analysis extends to other physical dimensions $d\le 3$ but we choose the prototypical
case $d=2$ to simplify the presentation. 
For simplicity we consider mean zero initial data, that is 
\begin{align}
\int_{\mathbb T^2} u_0(x) dx =0.
\end{align}
 For smooth solutions, there is  the mass conservation law
\begin{equation}
\frac d {dt} M(t) = \frac d {dt} \int_{\Omega} u(t,x) dx  \equiv 0.
\end{equation}
It follows that $u(t,\cdot)$ has zero mean for all $t>0$.  For the class of mean-zero functions with
suitable regularity,   one can employ the
operator $|\nabla|^s$ for $s<0$ as the Fourier multiplier $|k|^s \cdot 1_{k\ne 0}$. 
The system \eqref{1} naturally arises as a gradient flow of a Ginzburg-Landau type energy
functional $\mathcal E(u)$ in $H^{-1}$, namely
\begin{equation}
\partial_t u= - \frac {\delta \mathcal E} {\delta u } \Bigr|_{H^{-1}} =
\Delta (\frac {\delta \mathcal E} {\delta u} ),
\end{equation}
where $\frac {\delta \mathcal E}{\delta u} \Bigr|_{H^{-1}}$ , $\frac {\delta \mathcal E}
{\delta u}$ denote the standard variational derivatives in $H^{-1}$ and $L^2$ respectively, and
\begin{equation}
\mathcal E(u)= \int_{\Omega} \left( \frac 12 \nu |\nabla u|^2 + F(u) \right) dx
=\int_{\Omega} \left( \frac 12 \nu |\nabla u|^2 + \frac 14 (u^2-1)^2 \right) dx.
\end{equation}
For smooth solutions, the fundamental energy conservation law takes the form
\begin{equation}
\frac d {dt} \mathcal E ( u(t) ) +
\| |\nabla|^{-1}  \partial_t u \|_2^2
=\frac d {dt} \mathcal E(u(t)) + \int_{\Omega}
| \nabla ( -\nu \Delta u + f(u ) ) |^2 dx =0.
\end{equation}
 It follows that 
 \begin{subequations}
\begin{align}
&\mathcal E(u(t) ) \le \mathcal E(u(s)), \qquad \forall\, t\ge s;\\
& \| \nabla u(t) \|_2 \le \sqrt{\frac 2 {\nu}}  \mathcal E(u(t))
 \le \sqrt{\frac 2 {\nu}} \mathcal E(u_0 ), \qquad \forall\, t>0.
\end{align}
\end{subequations}
In particular, one obtains a priori $\dot H^1$-norm control of the solution for all $t>0$. 
Since the scaling-critical space for CH is $L^2$ in 2D, the global wellposedness and regularity 
for $H^1$-initial data follows easily. 

For $\tau>0$, we let $S_L(\tau)= e^{-\tau \nu \Delta^2}$ be the exact solution operator to the linear equation:
\begin{align}
\partial_t  u= - \nu \Delta^2 u.
\end{align}
We define $S_N(\tau): w \to u$ as the solution operator to the following problem:
\begin{align}
\frac{u-w} {\tau } = \Delta ( w^3 - w).
\end{align}
In yet other words,
\begin{align}
u =S_N(\tau) w = w+ \tau \Delta (w^3-w).
\end{align}
This is one of the simplest discretization on the timer interval $[0, \tau]$  for the exact problem 
\begin{align} \label{1.11}
\begin{cases}
\partial_t  u = \Delta ( u^3 -u), \qquad \quad t>0; \\
u \bigr|_{t=0} =w.
\end{cases}
\end{align}
By using the operator-splitting, the solution of the original equation from time $t$ to
time $t+\tau$ is approximated as
\begin{align}
u(t+\tau, x) 
\approx 
\Bigl(  S_L(  \tau ) S_N( \tau)  u \Bigr) (t,x).
\end{align}

The main purpose of this note is establish stability of the above operator-splitting
algorithm. Prior to our work, there were very few rigorous results on the analysis of the operator-splitting
type algorithms for the Cahn-Hilliard equation and similar models.  In \cite{Feng19},
Weng, Zhai and Feng considered a viscous Cahn-Hilliard model of the form
\begin{align}
(1-\alpha) \partial_t u = \Delta ( - \epsilon^2 \Delta u + f(u ) + \alpha \partial_t u),
\end{align}
where $0<\alpha<1$.  They considered a fast explicit Strang splitting and established
stability and convergence under the assumption that $A=\|\nabla u^{\operatorname{num}}\|^2_{\infty}$, $B=
\| u^{\operatorname{num}} \|_{\infty}^2$ are bounded,  and satisfy a technical condition
$6A+8-24B>0$ (see Theorem 1 on pp. 7 of \cite{Feng19}), where $u^{\mathrm{num}}$ denotes
the numerical solution. In \cite{Red19}, Gidey and Reddy considered a convective Cahn-Hilliard
model of the form
\begin{align} \label{1.14}
\partial_t u - \gamma \nabla \cdot \mathbf{h}(u) + \epsilon^2 \Delta^2 u
=\Delta (f(u)),
\end{align}
where $\mathbf{h}(u) =\frac 12 (u^2, u^2)$. They considered operator-splitting
of \eqref{1.14} into hyperbolic part, nonlinear diffusion part and diffusion part respectively,
and obtained various conditional results  concerning certain weak solutions.
In \cite{Tang15}, Cheng, Kurganov, Qu and Tang considered the Strang splitting
for the Cahn-Hilliard equation and molecular beam epitaxy type models.  Some conditional results were given
in \cite{Tang15} but rigorous analysis of energy stability has remained open.
The purpose of this note is to establish a new theoretical framework for the rigorous
analysis of energy stability and higher-order Sobolev-norm stability for the operator-splitting
method applied to these difficult equations.
 Our first result establishes uniform Sobolev control of the numerical solution for
all time.

\begin{thm}\label{thm0}
Let $\nu>0$ and consider the two-dimensional periodic torus $\mathbb T^2
=[-\pi, \pi]^2$.  Assume the initial data $u^0
\in H^{k_0}(\mathbb T^2)$ ($k_0\ge 1$ is an integer) and has mean zero.  Let $\tau>0$ and define
\begin{align}
u^{n+1} = S_L( {\tau}) S_N(\tau )  u^n, \quad n\ge 0.
\end{align}
There exists a constant $\tau_*>0$ depending only on $\|u^0\|_{H^1}$ and
$\nu$, such that if $0<\tau <\tau_*$, then
\begin{align}
\sup_{n\ge 0} \| u^n \|_{H^{k_0}} \le A_1<\infty,
\end{align}
where $A_1>0$ depends on ($\| u^0\|_{H^{k_0} }$, $\nu$, $k_0$). 
\end{thm}
\begin{rem}
Similar statements also hold if we consider
$u^{n+1} = S_N(\tau) S_L(\tau) u^n$. Theorem \ref{thm0} is a special case
of Theorem \ref{thm3.2} in Section 3.
\end{rem}

Our second result establishes the convergence of the operator splitting approximation.

\begin{thm}[Convergence of the splitting approximation] \label{thm1}
Assume the initial data $u^0 \in H^8(\mathbb T^2)$ with mean zero. 
Let $u^n$ be defined as in Theorem \ref{thm0}.  Let $u$ be the exact PDE solution
to \eqref{1}  corresponding to initial data $u^0$. Let $0<\tau <\tau_*$ as in Theorem
\ref{thm0}. Then for any $T>0$, we have
\begin{align}
\sup_{n\ge 1, n\tau \le T}  \| u^n - u(n\tau, \cdot ) \|_{L^2(\mathbb T^2)}
\le C \cdot \tau,
\end{align}
where $C>0$ depends on ($\nu$, $\|u^0\|_{H^8}$, $T$).
\end{thm}
\begin{rem}
The regularity assumption on initial data can be lowered but we shall not dwell on this
issue here for simplicity of presentation. One can also work out the convergence in
higher Sobolev norms. We shall not pursue this issue here.
\end{rem}

The rest of this note is organized as follows. In Section $2$ we set up the notation and collect
some preliminary lemmas.  In Section $3$ we analyze in detail the propagator $S_L(\tau)
S_N(\tau)$ and prove Theorem \ref{thm3.2}. Theorem \ref{thm0} follows as a special
case of Theorem \ref{thm3.2}. In Section $4$ we give the proof of Theorem \ref{thm1}.

\section{Notation and preliminaries}

For any two positive quantities $X$ and $Y$, we shall write $X\lesssim Y$ or $Y\gtrsim X$ if
$X \le  CY$ for some  constant $C>0$ whose precise value is unimportant.
We shall write $X\sim Y$ if both $X\lesssim Y$ and $Y\lesssim X$ hold.
We write $X\lesssim_{\alpha}Y$ if the
constant $C$ depends on some parameter $\alpha$.
We shall
write $X=O(Y)$ if $|X| \lesssim Y$ and $X=O_{\alpha}(Y)$ if $|X| \lesssim_{\alpha} Y$.

We shall denote $X\ll Y$ if
$X \le c Y$ for some sufficiently small constant $c$. The smallness of the constant $c$ is
usually clear from the context. The notation $X\gg Y$ is similarly defined. Note that
our use of $\ll$ and $\gg$ here is \emph{different} from the usual Vinogradov notation
in number theory or asymptotic analysis.

For any $x=(x_1,\cdots, x_d) \in \mathbb R^d$, we denote $|x| =|x|_2=\sqrt{x_1^2+\cdots+x_d^2}$, and
$|x|_{\infty} =\max_{1\le j \le d} |x_j|$.
Also occasionally we use the Japanese bracket notation:
$\langle x \rangle =(1+|x|^2)^{\frac 12}.$

We denote by $\mathbb T^d=[-\pi, \pi]^d = \mathbb R^d/2\pi \mathbb Z^d$ the usual
$2\pi$-periodic torus.
For $1\le p \le \infty$ and any function $f:\, x\in \mathbb T^d \to \mathbb R$, we denote
the Lebesgue $L^p$-norm of $f$ as
\begin{align*}
\|f \|_{L^p_x(\mathbb T^d)} =\|f\|_{L^p(\mathbb T^d)} =\| f \|_p.
\end{align*}
If $(a_j)_{j \in I}$ is a sequence of complex numbers
and $I$ is the index set, we denote the discrete $l^p$-norm
as
\begin{equation}
\| (a_j) \|_{l_j^p(j\in I)} = \| (a_j) \|_{l^p(I)} =
\begin{cases}
 {\displaystyle \left(\sum_{j\in I} |a_j|^p\right)^{\frac 1p}}, \quad 0<p<\infty, \\
 \sup_{j\in I} |a_j|, \quad \qquad p=\infty.
 \end{cases}
 \end{equation}
 For example,
$ \| \hat f(k) \|_{l_k^2(\mathbb Z^d)} = \left(\sum_{k \in \mathbb Z^d} |\hat f(k)|^2\right)^{\frac 12}$.
If $f=(f_1,\cdots,f_m)$ is a vector-valued function, we denote
$|f| =\sqrt{\sum_{j=1}^m |f_j|^2}$, and
$\| f\|_p = \| ({\sum_{j=1}^m f_j^2})^{\frac 12} \|_p$.
We use similar convention for the corresponding discrete $l^p$ norms for the vector-valued
case.


We use the following convention for the Fourier transform pair:
\begin{equation} \label{eqFt2}
\hat f(k) = \int_{\mathbb T^d} f(x) e^{- i k\cdot x} dx, \quad
 f(x) =\frac 1 {(2\pi)^d}\sum_{k\in \mathbb Z^d} \hat f(k) e^{ ik \cdot x},
\end{equation}
and denote for $0\le s \in \mathbb R$,
\begin{subequations}
\begin{align}
&\|f \|_{\dot H^s} = \|f \|_{\dot H^s(\mathbb T^d)} = \| |\nabla|^s f \|_{L^2(\mathbb T^d)}
\sim \|  |k|^s \hat f (k) \|_{l^2_k (\mathbb Z^d)}, \\
& \| f \|_{H^s} = \sqrt{\| f \|_2^2 + \| f\|_{\dot H^s}^2}  \sim \| \langle
 |k| \rangle^s \hat f(k) \|_{l^2_k(\mathbb Z^d)}.
\end{align}
\end{subequations}

\begin{lem} \label{leKbeta}
Let $d\le 3$ and $\beta>0$. Consider on the torus $\mathbb T^d=[-\pi, \pi]^d$,
\begin{equation}
K(x) = \mathcal F^{-1} ( e^{-\beta |k|^4} )
=e^{-\beta \Delta^2} \delta_0,
\end{equation}
 where $\delta_0$ is the periodic Dirac comb. Then for any $1\le p\le \infty$,
\begin{align}
\| K \|_{L^p(\mathbb T^d)} \le c_{d,p}\,  (1+\beta^{-d(\frac 14 -\frac 1{4p})}),
\end{align}
where $c_{d,p}>0$ depends only on $d$ and $p$.  Define
\begin{align}
\widetilde{K}= \mathcal F^{-1} ( e^{-\beta |k|^4} 1_{k\ne 0} ).
\end{align}
Then
\begin{align}
\| \widetilde{K} \|_{L^p(\mathbb T^d)} \le \tilde c_{d,p}\,  \beta^{-d(\frac 14 -\frac 1{4p})},
\end{align}
where $\tilde c_{d,p}>0$ depends only on $d$ and $p$.
\end{lem}
\begin{rem}
Define $K_w(x) = (2\pi)^{-d}  \int_{\mathbb R^d}
e^{ i \xi \cdot x} e^{-\beta |\xi|^4} d\xi$. By the usual Poisson summation formula,
it is not difficult to check that
\begin{align} \label{2.8Ta}
K(x)  =\sum_{l \in \mathbb Z^d} K_w( x+ 2\pi l).
\end{align}
This identity will be used below without explicit mentioning. We note that a formal proof 
of \eqref{2.8Ta} may proceed as follows. 
\begin{align}
 & (2\pi)^{-d} \sum_{k \in \mathbb Z^d} e^{-\beta |k|^4} e^{i k \cdot x} \notag \\
 = & (2\pi)^{-d} \sum_{k \in \mathbb Z^d}
 \int_{\mathbb R^d} K_w(y) e^{i k\cdot (x-y) } dy \notag \\
 =& \int_{\mathbb R^d}
 K_{w}(y) \sum_{l\in \mathbb Z^d} \delta(x-y- 2\pi l) d y  =
 \sum_{l\in \mathbb Z^d} K_w (x+ 2\pi l). \notag
 \end{align}
The above formal computation can be justified by the usual limiting process. We omit the 
details.
\end{rem}

\begin{proof}[Proof of Lemma \ref{leKbeta}]

Define
\begin{align*}
K_1(x) = \frac 1 {(2\pi)^d} \int_{\mathbb R^d} e^{i \xi \cdot x} e^{-|\xi|^4} d\xi.
\end{align*}
It is easy to check that $|K_1(x)| \lesssim \, \langle x \rangle^{-10}$ and $K_1 \in
L_x^1(\mathbb R^d)$ for $d\le 3$.
Now note that for $d\le 3$, if $|x|_{\infty} \le \pi$, then $|x| \le \sqrt d \pi
\le \sqrt 3 \pi$. Thus if $|l|\ge 400$, then
\begin{align*}
&\pi  |l| \le |x+ 2\pi l| \le 4\pi |l|, \qquad \forall\,  |x|_{\infty} \le \pi.
\end{align*}
It follows that for all $1\le p\le \infty$ and $|l|\ge 400$,
\begin{align*}
\| \langle \beta^{-\frac 14} (x+ 2\pi l) \rangle^{-10}
\|_{L_x^p(|x|_{\infty} \le \pi)} \lesssim
\langle \beta^{-\frac 14} \pi |l| \rangle^{-10}.
\end{align*}
Clearly then
\begin{align}
\| K\|_{L^p(\mathbb T^d)} &\le \beta^{-\frac d 4} \sum_{l \in \mathbb Z^d}
\| K_1( \beta^{-\frac 14}( x+2\pi l) ) \|_{L_x^{p} (|x|_{\infty} \le \pi )} \notag \\
& \lesssim \;
\beta^{-\frac d4}
\sum_{|l|\le 400} \| K_1 (\beta^{-\frac 14} (x+2\pi l)) \|_{p}
+ \sum_{|l|>400} \beta^{-\frac d4} \langle \beta^{-\frac 14} \pi |l| \rangle^{-10} \notag \\
& \lesssim \; \beta^{-d(\frac 14 - \frac 1 {4p} )} +1.  \label{tpK_e1}
\end{align}
Now we consider the estimate for $\widetilde{K}(x) = K(x)-\frac 1 {(2\pi)^d}$.
Obviously by using the previous bound we have $\|\widetilde{K}\|_1 \lesssim \|K\|_1+1
\lesssim 1$. Alternatively
one can compute
\begin{align*}
\| \tilde  K \|_{L_x^1(\mathbb T^d)} \lesssim  1+
\| \sum_{l \in \mathbb Z^d} \beta^{-\frac d 4} |K_1 (\beta^{-\frac 14} (x+2\pi l) )|
\|_{L_x^1(\mathbb T^d)} \lesssim  1+
\|  K_1 \|_{L_x^1(\mathbb R^d)} \lesssim 1.
\end{align*}
We  bound the $L^2$ norm as
\begin{align*}
\| \widetilde{K}\|_{L_x^2(\mathbb T^d)} \lesssim \| e^{-\beta|k|^4}  \|_{l_k^2(
0\ne k \in \mathbb Z^d)}
\lesssim\beta^{-\frac d8}.
\end{align*}
Similarly,
\begin{align*}
\| \widetilde{K}\|_{L_x^{\infty} (\mathbb T^d)} &\lesssim \| e^{-\beta |k|^4}  \|_{l_k^1(
0\ne k \in \mathbb Z^d)}
\lesssim \beta^{-\frac d4}.
\end{align*}
By using interpolation we then get the $L^p$ estimate.
\end{proof}

\begin{lem} \label{lem2.2}
Let  $d=2$ and $\nu>0$. Let $\tau>0$. Then for any $g\in L^4(\mathbb T^2)$
 with zero mean,
we have
\begin{align}
&\| e^{-\nu \tau \Delta^2}  g \|_{\infty} \le C_1 (\nu \tau)^{-\frac 18}
\|g\|_4;
\end{align}
For any $g_1\in L^{\frac 43}(\mathbb T^2)$, we have
\begin{align}
& \| \tau \Delta e^{-\nu \tau \Delta^2} g_1\|_{\infty}
\le C_2  \tau (\nu \tau)^{-\frac 78} \| g_1\|_{\frac 43}.
\end{align}
In the above $C_1>0$, $C_2>0$ are absolute constants.
\end{lem}

\begin{proof}
Denote $\beta=\nu \tau$.  The first inequality follows from Lemma
\ref{leKbeta} (see the bound for $\tilde K$ therein). For the second inequality denote
\begin{align}
K_{\beta}=\mathcal F^{-1}\left( \beta^{\frac 12} |k|^2 e^{-\beta |k|^4}  \right).
\end{align}
We then have $\| K_{\beta} \|_{L_x^4(\mathbb T^2)}\lesssim \| \widehat{K_{\beta} } \|_{l_k^{\frac 43} (\mathbb Z^2)} \lesssim
\; \beta^{-\frac 38}$.
\end{proof}

\begin{lem} \label{lem2.3}
Let $d\ge 1$.
If $E_p= \int_{\mathbb T^d} \frac 14 (v^2-1)^2 dx$, then
\begin{align}
\| v\|_{L^4(\mathbb T^d)} \lesssim 1+{E_p}^{\frac 14},\qquad
\| v^3 -v \|_{L^{\frac 43}(\mathbb T^d)} \lesssim E_p^{\frac 12} (1+ E_p^{\frac 14}).
\end{align}
\end{lem}

\begin{proof}
Obvious.
For the second inequality, note that
$\| (v^2-1)v \|_{\frac 43}\le \| v^2-1\|_2 \|v\|_4$.
\end{proof}

\section{Analysis of the propagator $S_L(\tau) S_N (\tau)$ }
In this section we analyze in detail the propagator $S_L(\tau) S_N(\tau)$. 
If $u=S_L (\tau) S_N(\tau) w$, then
\begin{align} \label{A3.1}
u = e^{-\tau \nu \Delta^2} \Bigl( w + \tau \Delta (w^3-w) \Bigr).
\end{align}
Denote 
\begin{align}
E_1(w) = \frac 1 {2\tau} \| |\nabla|^{-1}
(e^{\tau \nu \Delta^2} -1)^{\frac 12} w\|_2^2 + \frac 1 4 \int_{\mathbb T^2} (w^2-1)^2 dx.
\end{align}

\begin{thm}[One-step energy stability] \label{thm3.1}
Suppose $w$ has mean zero and $E_1(w) $ is finite. 
We have 
\begin{align}
E_1(u) -E_1(w)
 +\left(\frac 12+\sqrt{\frac{2\nu}{\tau}} \right) \|u-w\|_2^2 
\le   \|u-w \|_2^2 \cdot \frac 32 \operatorname{max}
   \{ \|u\|_{\infty}^2, \;
 \| w \|_{\infty}^2 \}.
 \end{align}
\end{thm}
\begin{proof}
Recall $f(w)= w^3-w$.
We rewrite \eqref{A3.1} as
\begin{align}
\frac {u-w}{\tau} + \frac {e^{\tau \nu \Delta^2} u -u} {\tau}
=\Delta (f(w)).
\end{align}

Taking the $L^2$ inner product with $(-\Delta)^{-1}(u-w)$ on both sides and
applying the identity $ b\cdot (b-a) = \frac 12 ( |b|^2-|a|^2+|b-a|^2)$,
we get
\begin{eqnarray}
  && \frac 1 {\tau} \| |\nabla|^{-1} (u-w) \|_2^2
  + \frac 12\Big( \| T u \|_2^2 -\|T w\|_2^2
+ \| T(u-w) \|_2^2 \Big) \nonumber \\
& =& ( \Delta ( f(w) ), (-\Delta)^{-1} (u-w) ),
\end{eqnarray}
where $T=|\nabla|^{-1} \tau^{-\frac 12} (e^{\tau \nu\Delta^2} -1)^{\frac 12}$. 
Clearly
\begin{align} \notag
 (\Delta ( f(w) ), (-\Delta)^{-1}(u-w) ) = - (f(w), u-w).
 \end{align}
Introduce the auxiliary function $g(s)=F(w+s(u-w))$, where $F(z)=\frac 14 (z^2-1)^2$.  By using the Taylor expansion
$g(1)=g(0)+g^{\prime}(0) + \int_0^1 g^{\prime\prime}(s) (1-s)ds$,
we get
\begin{align}
F(u) =& F(w) +f(w) (u-w) -\frac 1 2 (u-w)^2 \notag \\
& \;\; + (u-w)^2 \int_0^1 \tilde f^{\prime}( w+s(u-w) )  (1-s)ds,
\end{align}
where $\tilde f(z)=z^3$ and $\tilde f^{\prime}(z)=3z^2$ for $z\in \mathbb R$. From this it is easy to see
that
\begin{align}
&  - (f(w), u-w) \notag \\
 \le&  F(w) - F(u) -\frac 1 2 \|u-w\|_2^2
+\|u-w\|_2^2 \cdot \frac 32
\max\{\|u\|_{\infty}^2,\|w\|_{\infty}^2\}.
\end{align}
Thus
\begin{align}
  & E_1(u)-E_1(w)+ \frac 1 {\tau} \| |\nabla|^{-1} (u-w) \|_2^2
+\frac{1}2 \| T(u-w) \|_2^2  + \frac 12 \|u-w \|_2^2  \notag \\
\le& \|u-w\|_2^2 \cdot \frac 32
\max\{\|u\|_{\infty}^2,\|w\|_{\infty}^2\}.
\end{align}
Now observe that for $\xi \ne 0$, 
\begin{align}
\frac 1 {\tau |\xi|^2}
+ \frac {e^{\tau \nu |\xi|^4} -1} {2\tau |\xi|^2}
=\frac {e^{\tau \nu |\xi|^4} +1} {2\tau |\xi|^2}
\ge \frac {2+\tau \nu |\xi|^4} {2\tau |\xi|^2} 
\ge \sqrt{\frac {2\nu} {\tau} }.
\end{align}
It follows that 
\begin{align}
\frac 1 {\tau} \| |\nabla|^{-1} (u-w) \|_2^2
+\frac 12 \| T(u-w)\|_2^2 \ge \sqrt{\frac {2\nu}{\tau}}
\| u-w\|_2^2.
\end{align}
The desired inequality  follows easily.
\end{proof}

\begin{lem} \label{lem3.1}
We have
\begin{align} \label{3.11}
\| u \|_{\infty}
\le c_1 \cdot (\nu \tau)^{-\frac 18}
(1+E_1(w)^{\frac 14}) +
c_1 \cdot \tau (\nu \tau)^{-\frac 78} E_1(w)^{\frac 12}
(1+E_1(w)^{\frac 14} ),
\end{align}
where $c_1>0$ is an absolute constant. Assume
\begin{align}
\| w\|_{\infty} \le \alpha_1 (\nu \tau)^{-\frac 18} + \alpha_2
\tau (\nu \tau)^{-\frac 78},
\end{align}
for some constants $\alpha_1$, $\alpha_2$ satisfying
\begin{align}
\alpha_1\ge c_1 (1+E_1(w)^{\frac 14}), 
\qquad \alpha_2\ge c_1
\cdot E_1(w)^{\frac 12} (1+E_1(w)^{\frac 14} ).
\end{align}
Define $\alpha_*= \max\{\alpha_1, \, \alpha_2 \}$. If
\begin{align}
0<\tau<   c \cdot \min\{ \alpha_*^{-8}, \alpha_*^{-\frac 83} \} \nu^3,
\end{align}
where $c>0$ is a sufficiently small absolute constant, then 
\begin{align} \label{3.15}
E_1(u) \le E_1(w).
\end{align}
\end{lem}
\begin{proof}
The bound \eqref{3.11} follows from Lemma \ref{lem2.2} and Lemma \ref{lem2.3}.
To show \eqref{3.15}, by Theorem \ref{thm3.1}, we only need to check the inequality
\begin{align}
\frac 1 2 +\sqrt{\frac {2\nu}{\tau}}
\ge \frac 32 \cdot \max\{ \|u \|_{\infty}^2, \|w\|_{\infty}^2\}.
\end{align}
It amounts to checking the inequalities
\begin{align}
\sqrt{\frac {2\nu}{\tau}}
\gg \alpha_*^2 (\nu \tau)^{-\frac 14}, \quad
\sqrt{\frac {2\nu}{\tau}} \gg \alpha_*^2 \tau^2 (\nu \tau)^{-\frac 74}.
\end{align}
The result is obvious.
\end{proof}

The following lemma shows that the energy $E_1(w)$ is well-defined.
\begin{lem} \label{lem3.2}
Suppose $u^0 \in H^1(\mathbb T^2)$ and has mean zero.  Set $w=S_L({\tau} ) 
S_N(\tau) u^0$.
Then
\begin{align}
& E_1(w) \le c_0^{(1)} (1+ \nu +\nu^{-1})^4 ( 1+ \|u^0\|_{H^1}^3)^4; \notag \\
& \| w\|_{\infty} \le  c_0^{(2)} (\nu \tau)^{-\frac 18} 
(1+\nu^{-1}) (\| u^0 \|_{H^1}+ \|u^0\|_{H^1}^3),
\end{align}
where $c_0^{(1)}>0$, $c_0^{(2)}>0$ are absolute constants.
\end{lem}
\begin{proof}
First we note that
\begin{align}
\frac 14 \int (w^2-1)^2 dx \lesssim 1 + \| w\|_4^4 \lesssim 1 + \| w\|_{H^1}^4.
\end{align}
Since $w= S_L(\tau) (u^0 +\tau \Delta ( f(u^0) ) )$, it follows that
(below $\overline{f(u^0)}$ denotes the average of $f(u^0)$ on $\mathbb T^2$)
\begin{align}
\| w\|_{H^1} &\lesssim \| u^0\|_{H^1} + \| \tau \Delta |\nabla| e^{-\nu \tau \Delta^2} ( f(u^0) ) \|_2 \notag \\
&\lesssim \| u^0 \|_{H^1} + \| \tau |\nabla|^{3.5} e^{-\nu \tau \Delta^2}
( f(u^0) - \overline{f(u^0) } ) \|_{\frac 43}  \notag \\
& \lesssim \| u^0\|_{H^1} + \nu^{-1}  (\|u^0\|_{H^1} + \| u^0 \|_{H^1}^3). 
\end{align}
Write $w= S_L(\frac {\tau} 2) g$, where $g= S_L(\frac {\tau} 2) ( u^0 +
\tau \Delta (f (u^0) ) )$. By a similar estimate as above, we have
\begin{align}
\| g\|_{H^1} \lesssim 
\| u^0\|_{H^1} + \nu^{-1}  (\|u^0\|_{H^1} + \| u^0 \|_{H^1}^3). 
\end{align}
Clearly
\begin{align}
\frac 1 {2\tau}
(|\nabla|^{-2}(e^{\tau \nu \Delta^2} -1) w, w )
& = \nu ( \frac {1- e^{-\tau \nu \Delta^2} } {2\nu\tau \Delta^2} |\nabla | g, |\nabla | g ) 
\notag \\
& \lesssim \nu \| g \|_{H^1}^2.
\end{align}
The desired bound on $E_1(w)$ easily follows.

For the $L^{\infty}$-bound, we note that by Lemma \ref{lem2.2},
\begin{align}
\| w\|_{\infty} = \| e^{-\frac {\tau}2 \nu \Delta^2} g
\|_{\infty} 
\lesssim (\nu \tau)^{-\frac 18} \| g\|_4
\lesssim (\nu \tau)^{-\frac 18} \| g \|_{H^1}.
\end{align}
\end{proof}

 Theorem \ref{thm0} is a simplified version of the following Theorem. 

\begin{thm} \label{thm3.2}
Suppose $u^0 \in H^1(\mathbb T^2)$ and has mean zero.  Define
$u^1= S_L( \tau) S_N(\tau) u^0$. Then
\begin{align}
& E_1(u^1) \le c_0^{(1)} (1+ \nu +\nu^{-1})^4 ( 1+ \|u^0\|_{H^1}^3)^4; \notag \\
& \| u^1\|_{\infty} \le  c_0^{(2)} (\nu \tau)^{-\frac 18} 
(1+\nu^{-1}) (\| u^0 \|_{H^1}+ \|u^0\|_{H^1}^3),
\end{align}
where $c_0^{(1)}>0$, $c_0^{(2)}>0$ are absolute constants and we recall
\begin{align}
E_1( u^1 ) = \frac 1 {2\tau} \| |\nabla|^{-1}
(e^{\tau \nu \Delta^2} -1)^{\frac 12} u^1\|_{L^2(\mathbb T^2)}^2 + \frac 1 4 \int_{\mathbb T^2} ((u^1)^2-1)^2 dx.
\end{align}
Set
\begin{align}
\alpha=\max\Bigl\{ c_1 (1+E_1(u^1)^{\frac 14}),
\; c_1 E_1(u^1)^{\frac 12} (1+E_1(u^1)^{\frac 14} ),\;
c_0^{(2)} (1+\nu^{-1}) (\| u^0 \|_{H^1}+ \|u^0\|_{H^1}^3) \Bigr\},
\end{align}
where $c_1$ is the same absolute constant in \eqref{3.11}.   Define the iterates
\begin{align}
u^{n+1} = S_L(\tau) S_N(\tau) u^{n}, \quad n\ge 1.
\end{align}
If $0<\tau < \tau_* = c \cdot \min\{ \alpha^{-8}, \alpha^{-\frac 83} \} \nu^3$ 
where $c>0$ is a sufficiently small absolute constant, then
it holds that
\begin{align}
& E_1(u^{n+1} ) \le E_1(u^n), \qquad \forall\, n\ge 1; \\
& \sup_{n\ge 0} \|u^n \|_{\infty}
\le \alpha (\nu \tau)^{-\frac 18} + \alpha \tau (\nu \tau)^{-\frac 78}.
\end{align}
Furthermore, if $u^0  \in {H^{k_0}(\mathbb T^2)} $ for some integer $k_0\ge 2$, then
we also have the uniform $H^{k_0}$ bound:
\begin{align} \label{3.28}
\sup_{n\ge 0} \| u^n \|_{H^{k_0} (\mathbb T^2)} \le B_1<\infty,
\end{align}
where $B_1>0$ depends on ($\| u^0 \|_{H^{k_0}(\mathbb T^2)}$, $\nu$, $k_0$). 
\end{thm}
\begin{proof}
The estimates of $u^1$ follows from Lemma \ref{lem3.2}. The energy decay and $L^{\infty}$
bound on $u^n$ follows from Lemma \ref{lem3.1} and an induction argument. 
For 
\eqref{3.28}, we note that
\begin{align}
u^{n+1} & =S_L(\tau) S_N (\tau ) u^n \notag \\
& = S_L(\tau) ( u^n + \tau \Delta ( f(u^n ) ) ) \notag \\
& = S_L(\tau) (S_L(\tau) u^{n-1} + \tau \Delta ( f(u^{n-1} ) ) ) + \tau S_L(\tau) \Delta
(f (u^n) ) \notag \\
& = \cdots \notag \\
& = S_L( (n+1)\tau) u^0 + \tau \sum_{k=0}^n  S_L( (k+1)\tau)  \Delta ( f(u^{n-k} ) ). 
\end{align}
The desired estimate then follows from the above using smoothing estimates (cf. \cite{lt2021}).
We omit
the details.
\end{proof}

\section{Proof of Theorem \ref{thm1}}
In this section we complete the proof of Theorem \ref{thm1}.
For convenience we shall set $\nu=1$. 
Since $u^{n+1} = S_L( {\tau}) S_N(\tau) u^n$, we have
\begin{align}
u^{n+1} 
& = e^{-\tau \Delta^2} u^n 
+\tau e^{- \tau \Delta^2} \Delta ( f(  u^n ) ).
\end{align}
We rewrite the above as
\begin{align}
u^{n+1} = (1+\tau \Delta^2)^{-1} u^n + \tau (1+\tau \Delta^2)^{-1} \Delta 
( f(u^n) ) +  (1+\tau \Delta^2)^{-1} g^n,  \label{4.2}
\end{align}
where
\begin{align}
g^n&= (1+\tau \Delta^2)
\Bigl( 
(e^{-\tau\Delta^2} - (1+\tau \Delta^2)^{-1}) u^n
 + \tau( e^{- \tau \Delta^2} -(1+\tau \Delta^2)^{-1} ) \Delta ( f (u^n) ) \Bigr).
\end{align}

\begin{lem} \label{lem4.1}
For some absolute constant $d_1>0$, it holds that
\begin{align}
\| g^n\|_2 \le  d_1 \tau^2 \cdot ( \| u^n \|_{H^8} + \| u^n \|_{H^8}^3).
\end{align}
\end{lem}
\begin{proof}
Obvious.
\end{proof}
Rewrite \eqref{4.2} as
\begin{align}
\frac {u^{n+1} -u^n} {\tau} = -\Delta^2 u^{n+1}
+ \Delta (f(u^n) ) + g^n.
\end{align}
Note that  $\sup_{n\ge 0} \| u^n \|_{H^8} \lesssim 1$.  With the help of Lemma \ref{lem4.1}, the proof of Theorem \ref{thm1}
then follows from Proposition 4.1 of \cite{LQT16}.


\frenchspacing
\bibliographystyle{plain}

\end{document}